\documentclass{amsart}%
\usepackage{amsfonts}
\usepackage{amsmath}
\usepackage{amssymb}
\usepackage{graphicx}%
\newtheorem{theorem}{Theorem}
\theoremstyle{plain}

\newtheorem{lemma}{Lemma}

\newtheorem{proposition}{Proposition}

\numberwithin{equation}{section}
\begin{document}
\title[Isoperimetric Constant]{On the Isoperimetric Constant of Symmetric Spaces of Noncompact Type}
\author{Xiaodong Wang}
\address{Department of Mathematics, Michigan State University, East Lansing, MI 48824}
\email{xwang@math.msu.edu}
\maketitle

\bigskip Let $N^{n}$ be a complete, noncompact Riemannian manifold. We
consider the isoperimetric constant $I\left(  N\right)  $ defined by%
\[
I\left(  N\right)  =\inf_{\Omega}\frac{A\left(  \partial\Omega\right)
}{V\left(  \Omega\right)  },
\]
where $\Omega$ ranges over open submanifolds of $N$ with compact closure and
smooth boundary. This is also called Cheeger's constant \cite{Ch}. The
importance of this global geometric invariant is illustrated by the following
fundamental inequality relating it to another global analytic invariant
\begin{equation}
\lambda_{0}\left(  N\right)  \geq I\left(  N\right)  ^{2}/4, \label{Cheeger}%
\end{equation}
where $\lambda_{0}$ is the bottom of spectrum of the Laplace operator on $N$.
It is a well known fact that $\lambda_{0}$ can be characterized as%
\begin{equation}
\lambda_{0}=\inf\frac{\int_{N}\left\vert \nabla f\right\vert ^{2}}{\int
_{N}\left\vert f\right\vert ^{2}}, \label{lam0}%
\end{equation}
where $f$ ranges over nonzero $C^{1}$ functions with compact support.

For $\mathbb{R}^{n}$ or any Riemannian manifold with nonnegative Ricci
curvature, the isoperimetric constant $I$ is trivial. On the other hand, for
Cartan-Hadamard manifolds with sectional curvature bounded by a negative
constant from above, Yau proved that $I$ is always positive.

\begin{proposition}
(Yau \cite{Y}) If $N$ is simply connected with sectional curvatures
$\leq\kappa<0$, then
\[
I\left(  N\right)  \geq\left(  n-1\right)  \sqrt{-k}.
\]

\end{proposition}

From this result one can easily deduce $I\left(  \mathbb{H}^{n}\right)  =n-1$.
For a detailed discussion of Cheeger's constant and related results, one can
consult \cite[Chapter 6]{Cha}. In general, it is very difficult to know if the
isoperimetric constant is positive or not and it is almost impossible to
compute it explicitly if it is known to be positive. In this short note, we
prove that the isoperimetric constant is positive for all symmetric spaces of
noncompact type and compute it explicitly.

Let $\left(  \widetilde{M}^{n},g\right)  $ be a Cartan-Hadamard manifold (i.e.
complete, simply-connected with nonpositive curvature) and $\mathcal{S}%
\widetilde{M}$ its unit tangent bundle. For any $p\in M$ and $u\in
\mathcal{S}_{p}\widetilde{M}$ we have a nonnegative symmetric operator
$R_{u}:T_{p}\widetilde{M}\rightarrow T_{p}\widetilde{M}$ defined by
\[
R_{u}\left(  X\right)  =-R\left(  u,X\right)  u.
\]
Let $0=\lambda_{0}\left(  u\right)  \leq\lambda_{2}\left(  u\right)
\leq\cdots\leq\lambda_{n-1}\left(  u\right)  $ be its eigenvalues. In this way
we have $n$ continuous functions $\lambda_{0},\cdots,\lambda_{n-1}$ on
$\mathcal{S}\widetilde{M}$. Obviously $\lambda_{i}\left(  -u\right)
=\lambda_{i}\left(  u\right)  $.

From now on we assume $\widetilde{M}$ is a symmetric space of noncompact type.
By this we mean that $\widetilde{M}$ is a Cartan-Hadamard manifold with
parallel curvature tensor and there is no Euclidean factor in its de Rham
decompostion. When $\widetilde{M}$ has rank one, it is negatively curved. But
if the rank is higher, its sectional curvature vanishes on certain $2$-planes.
A standard reference on symmetric spaces is Helgason \cite{H}.

We fix a base point $o\in\widetilde{M}$. For $\xi\in\mathcal{S}_{o}%
\widetilde{M}$ let $\gamma_{\xi}$ be the geodesic ray with initial velocity
$\xi$. We can choose an orthonormal basis $\left\{  e_{1},\cdots
,e_{n-1}\right\}  $ for $\xi^{\perp}$ s.t.
\[
R_{\xi}e_{i}=-R\left(  \xi,e_{i}\right)  \xi=\lambda_{i}\left(  \xi\right)
e_{i}.
\]
Let $E_{i}$ be the parallel vector field along $\gamma_{\xi}$ with
$E_{i}\left(  0\right)  =e_{i}$. Since the curvature tensor is parallel, we
have along $\gamma_{\xi}$
\[
R_{\gamma_{\xi}^{\prime}\left(  t\right)  }E_{i}\left(  t\right)  =-R\left(
\gamma_{\xi}^{\prime}\left(  t\right)  ,E_{i}\left(  t\right)  \right)
\gamma_{\xi}^{\prime}\left(  t\right)  =\lambda_{i}\left(  \xi\right)
E_{i}\left(  t\right)  .
\]
Therefore $\lambda_{i}\left(  \gamma_{\xi}^{\prime}\left(  t\right)  \right)
=\lambda_{i}\left(  \xi\right)  $. This proves that the $n$ continuous
functions $\lambda_{0},\cdots,\lambda_{n-1}$ on $\mathcal{S}\widetilde{M}$ are
invariant under the geodesic flow.

Along the geodesic $\gamma=$ $\gamma_{\xi}$ the Jacobi field equation%
\[
X^{\prime\prime}\left(  t\right)  +R\left(  \gamma^{\prime},X\right)
\gamma^{\prime}=0.
\]
can be explicitly solved. The solution satisfying the initial condition
$X\left(  0\right)  =0,X^{\prime}\left(  0\right)  =e_{i}$ is given by
\begin{equation}
X_{i}\left(  t\right)  =\frac{\sinh\sqrt{\lambda_{i}\left(  \xi\right)  }%
t}{\sqrt{\lambda_{i}\left(  \xi\right)  }}E_{i}\left(  t\right)  .
\label{Jacobi}%
\end{equation}

For any integer $k\geq1$, define the function $b_{k}$ on $\widetilde{M}$ by
$b_{k}\left(  x\right)  =d\left(  x,p_{k}\right)  -k$, where $p_{k}%
=\gamma_{\xi}\left(  k\right)  $. The Busemann function $b_{\xi}$ is the limit
of $b_{k}$ as $k\rightarrow\infty$, i.e.%
\[
b_{\xi}\left(  x\right)  =\lim_{k\rightarrow\infty}d\left(  x,p_{k}\right)
-k.
\]
It is well known that the limit exists. By \cite{HI}, the convergence is
locally uniform in $C^{2}\left(  \widetilde{M}\right)  $ and in particular
$b_{\xi}\in C^{2}\left(  \widetilde{M}\right)  $.

\begin{lemma}
\bigskip$\Delta b_{\xi}$ is constant and equals $l\left(  \xi\right)
:=\sum_{i}\sqrt{\lambda_{i}\left(  \xi\right)  }$.
\end{lemma}

\begin{proof}
We fix $x$ and denote $l_{k}=d\left(  x,p_{k}\right)  $. Let $\sigma
_{k}:\left[  0,l_{k}\right]  \rightarrow M$ be the geodesic from $x$ to
$p_{k}$. We write $u_{k}=-\gamma_{\xi}^{\prime}\left(  k\right)
,v_{k}=-\sigma^{\prime}\left(  l_{k}\right)  $. We have
\begin{equation}
\Delta b_{k}\left(  x\right)  =\sum_{i}\sqrt{\lambda_{i}\left(  v_{k}\right)
}\coth\sqrt{\lambda_{i}\left(  v_{k}\right)  }\left(  b_{k}\left(  x\right)
+k\right)  . \label{Dbk}%
\end{equation}
Let $\theta_{k}=\angle(u_{k},v_{k})$ be the angle between $u_{k}$ and $v_{k}$.
By the cosine law,%
\[
\cos\theta_{k}\geq\frac{k^{2}+l_{k}^{2}-d\left(  o,x\right)  ^{2}}{2kl_{k}}.
\]
As $\left\vert l_{k}-k\right\vert \leq d\left(  o,x\right)  $, it is obvious
that $\theta_{k}\rightarrow0$ as $k\rightarrow\infty$. For each $k$, there
exists $\phi_{k}\in G$ s.t. $\phi_{k}\left(  p_{k}\right)  =o$. Let
$\widetilde{u}_{k}=\phi_{\ast}\left(  u_{k}\right)  ,\widetilde{v}_{k}%
=\phi_{\ast}\left(  v_{k}\right)  $. They are unit vectors at $o$ and
\[
\angle(\widetilde{u}_{k},\widetilde{v}_{k})=\angle(u_{k},v_{k})\rightarrow0
\]
as $k\rightarrow\infty$. By continuity, for each $i$ we have%
\[
\left\vert \lambda_{i}\left(  \widetilde{u}_{k}\right)  -\lambda_{i}\left(
\widetilde{v}_{k}\right)  \right\vert \rightarrow0
\]
as $k\rightarrow\infty$. Since $\lambda_{i}\left(  \widetilde{u}_{k}\right)
=\lambda_{i}\left(  u_{k}\right)  =\lambda_{i}\left(  \xi\right)  $ and
$\lambda_{i}\left(  \widetilde{v}_{k}\right)  =\lambda_{i}\left(
v_{k}\right)  $, we have $\lambda_{i}\left(  v_{k}\right)  \rightarrow
\lambda_{i}\left(  \xi\right)  $. As $b_{k}\rightarrow b_{\xi}$ in
$C_{loc}^{2}$, we obtain from (\ref{Dbk}) by taking limit%
\[
\Delta b_{\xi}\left(  x\right)  =\sum_{i}\sqrt{\lambda_{i}\left(  \xi\right)
}.
\]

\end{proof}

Let $G=I_{0}(\widetilde{M})\mathfrak{\ }$be the connected component of the
isometry group containing the identity and $K=\left\{  \phi\in G:\phi
o=o\right\}  $. Thus $G$ is a semisimple Lie group acting transitively on
$\widetilde{M}$ by isometries and $K$ a maximal compact subgroup of $G$. Define

$\mathfrak{\ \ }$%
\begin{align*}
\mathfrak{g}  &  =\{\text{Killing vector fields on }M\},\\
\mathfrak{t}  &  =\{X\in\mathfrak{g}:X(o)=0\},\\
\mathfrak{\ \ p}  &  =\{X\in\mathfrak{g}:\nabla X(o)=0\}.
\end{align*}
We know that $\mathfrak{g}$ is the Lie algebra of $G\mathfrak{\ }$and
$\mathfrak{t}$ is the Lie algebra of $K$. Moreover
\[
\mathfrak{g}=\mathfrak{t}\oplus\mathfrak{p}%
\]
and $\mathfrak{p}$ is naturally identified with $T_{o}\widetilde{M}$. Let
$\sigma:\mathfrak{g}\rightarrow\mathfrak{g}$ be the Cartan involution, i.e. it
is the automorphism s.t. $\sigma|_{\mathfrak{t}}=I,\sigma|_{\mathfrak{p}}=-I$.
Let $B$ be the Killing form $\mathfrak{g}$, i.e. for $X,Y\in\mathfrak{g}$%
\[
B\left(  X,Y\right)  =\mathrm{tr}\left(  \mathrm{ad}_{X}\mathrm{ad}%
_{Y}\right)  .
\]
Since $\mathrm{ad}_{\sigma\left(  X\right)  }=\sigma\circ\mathrm{ad}_{X}%
\circ\sigma^{-1}$, $B$ is invariant under $\sigma$. In particular
$\mathfrak{t}$ and $\mathfrak{p}$ are orthogonal to each other w.r.t. the
Killing form of $\mathfrak{g}$, i.e. $B\left(  X,Y\right)  =0$ for any
$X\in\mathfrak{t},Y\in\mathfrak{p}$. Moreover, $B$ is negative definite on
$\mathfrak{t}$ and positive definite on $\mathfrak{p}$. Thus
\[
\left\langle X,Y\right\rangle =-B\left(  \sigma X,Y\right)
\]
is a metric on $\mathfrak{g}$. It is easy to verify that $\mathfrak{t}$ and
$\mathfrak{p}$ are still orthogonal and the $\mathrm{ad}$-action is
skew-symmetric w.r.t. this metric. To fix the scale, we assume that the
Riemannian metric on $T_{o}\widetilde{M}=\mathfrak{p}$ coincides with the
restriction of $\left\langle \cdot,\cdot\right\rangle $ on $\mathfrak{p}$.

\bigskip Let $\mathfrak{a}\subset\mathfrak{p}$ be a maximal abelian subspace.
For any $\alpha\in$ $\mathfrak{a}^{\ast}$ we define%

\[
\mathfrak{g}_{\alpha}=\{X\in\mathfrak{g}:\mathrm{ad}_{H}X=\alpha(H)X\text{
\ for all }H\in\mathfrak{a}\}.
\]
If $\mathfrak{\mathfrak{g}_{\alpha}}\neq0$, then $\alpha$ is called a root of
$\mathfrak{a}$ and $m_{\alpha}:=\dim\mathfrak{\mathfrak{g}_{\alpha}}$ is
called its multiplicity. The set of all nonzero roots is denoted by $\Delta$.
If $\alpha\in\Delta$, then $-\alpha\in\Delta$. Moreover, $\sigma$ defines an
isomorphism from $\mathfrak{g}_{\alpha}$ onto $\mathfrak{g}_{-\alpha}$. The
connected components of $\mathfrak{a}\backslash\cup_{\alpha\in\Delta}%
\ker\alpha$ are the Weyl chambers of $\mathfrak{a}$. Pick one of them to be
the positive Weyl chamber and denote it by $\mathfrak{a}^{+}$. A root is
positive if it is positive on $\mathfrak{a}^{+}$. Let $\Delta^{+}$ denote the
set of positive roots. We have the following orthogonal decomposition%

\[
\mathfrak{g}=\mathfrak{g}_{0}+\sum_{\alpha\in\Delta}\mathfrak{g}_{\alpha}.
\]
By definition we have $\mathfrak{a}\subset\mathfrak{g}_{0}$. In fact
$\mathfrak{g}_{0}=\mathfrak{g}_{0}\cap\mathfrak{t}\oplus\mathfrak{g}_{0}%
\cap\mathfrak{p}$. Since $\mathfrak{a}$ is maximal, $\mathfrak{a}%
=\mathfrak{g}_{0}\cap\mathfrak{p}$. Moreover $\left[  \mathfrak{g}_{\alpha
},\mathfrak{g}_{\beta}\right]  \subset\mathfrak{g}_{\alpha+\beta}$.

It is known that for any two maximal abelian subspaces $\mathfrak{a}%
,\widetilde{\mathfrak{a}}\subset\mathfrak{p}$ and Weyl chambers $\mathfrak{a}%
^{+}\subset\mathfrak{a},\widetilde{\mathfrak{a}}^{+}\subset\widetilde
{\mathfrak{a}}$ there exists $\phi\in K$ s.t. $\phi$ maps $\mathfrak{a}$ to
$\widetilde{\mathfrak{a}}$ and $\mathfrak{a}^{+}$ to $\widetilde{\mathfrak{a}%
}^{+}$. Therefore, we may assume $\xi\in\mathfrak{a}^{+}$. Consider the linear
map $T=\mathrm{ad}_{\xi}:\mathfrak{t}\rightarrow\mathfrak{p}$. and its adjoint
$T^{\ast}=\mathrm{ad}_{\xi}:\mathfrak{p\rightarrow t}$. Indeed, for
$u\in\mathfrak{t},v\in\mathfrak{p}$%
\begin{align*}
\left\langle Tu,v\right\rangle  &  =B\left(  \mathrm{ad}_{\xi}u,v\right)  \\
&  =-B\left(  u,\mathrm{ad}_{\xi}v\right)  \\
&  =\left\langle u,\mathrm{ad}_{\xi}v\right\rangle .
\end{align*}
How do we calculate those eigenvalues $\lambda_{i}\left(  \xi\right)  $?
Recall that they are the eigenvalues of the curvature operator $R_{\xi
}:\mathfrak{p}\rightarrow\mathfrak{p}$ defined by $R_{\xi}v=-R\left(
\xi,v\right)  \xi$. It is a basic formula for symmetric spaces that $R_{\xi
}v=-\left[  \left[  \xi,v\right]  ,\xi\right]  =\mathrm{ad}_{\xi}%
\mathrm{ad}_{\xi}v=TT^{\ast}v$. The curvature operator naturally extends to an
endomorphism on $\mathfrak{g}=\mathfrak{t}\oplus\mathfrak{p}$, to be denoted
by the same symbol $R_{\xi}$. In terms of the decomposition it is given by%
\[
\left[
\begin{array}
[c]{cc}%
0 & T\\
T^{\ast} & 0
\end{array}
\right]  ^{2}=\left[
\begin{array}
[c]{cc}%
TT^{\ast} & 0\\
0 & T^{\ast}T
\end{array}
\right]  .
\]
Therefore
\[
\sum_{i}\sqrt{\lambda_{i}\left(  \xi\right)  }=\mathrm{tr}\sqrt{R_{\xi}%
}|_{\mathfrak{p}}=\frac{1}{2}\mathrm{tr}\sqrt{R_{\xi}}.
\]
Then for any $v\in\mathfrak{g}$ we have decomposition%
\[
v=v_{0}+\sum_{\alpha\in\Delta}v_{\alpha}%
\]
Thus
\begin{align*}
R_{\xi}v &  =\mathrm{ad}_{\xi}\mathrm{ad}_{\xi}v\\
&  =\mathrm{ad}_{\xi}\left(  \sum_{\alpha}\alpha(\xi)v_{\alpha}\right)  \\
&  =\sum_{\alpha}\alpha(\xi)^{2}v_{\alpha}.
\end{align*}
Therefore%
\begin{align*}
\sum_{i}\sqrt{\lambda_{i}\left(  \xi\right)  } &  =\frac{1}{2}\mathrm{tr}%
\sqrt{R_{\xi}}\\
&  =\frac{1}{2}\sum_{\alpha\in\Delta}\left\vert \alpha(\xi)\right\vert
m_{\alpha}\\
&  =\sum_{\alpha\in\Delta^{+}}\alpha(\xi)m_{\alpha}.
\end{align*}
Let $e_{\alpha}\in\mathfrak{a}$ be the vector s.t. $\alpha\left(  X\right)
=\left\langle e_{\alpha},X\right\rangle $. Then%
\[
l\left(  \xi\right)  =\sum_{i}\sqrt{\lambda_{i}\left(  \xi\right)
}=\left\langle \xi,H\right\rangle ,
\]
where $H=\sum_{\alpha\in\Delta^{+}}m_{\alpha}e_{\alpha}\in\mathfrak{a}^{+}$.

\begin{lemma}
\label{low}We have
\[
I\left(  \widetilde{M}\right)  \geq\left\vert H\right\vert .
\]

\end{lemma}

\begin{proof}
For each $\xi\in\mathfrak{a}^{+}$, the corresponding Busemann function
$b_{\xi}$ satisfies $\Delta b_{\xi}=\left\langle \xi,H\right\rangle $. Then
for any open submanifold $\Omega\subset\widetilde{M}$ with compact closure and
smooth boundary, integrating over $\Omega$ yields
\begin{align*}
\left\langle \xi,H\right\rangle V\left(  \Omega\right)   &  =\int_{\Omega
}\Delta b_{\xi}dv\\
&  =\int_{\partial\Omega}\left\langle \nabla b_{\xi},\nu\right\rangle d\sigma,
\end{align*}
where $\nu$ is the outer unit normal of $\partial\Omega$. Since $\left\vert
\nabla b_{\xi}\right\vert \equiv1$, we obtain%
\[
\left\langle \xi,H\right\rangle V\left(  \Omega\right)  \leq A\left(
\partial\Omega\right)  .
\]
Therefore, for any $\xi\in\mathfrak{a}^{+}$
\[
I\left(  \widetilde{M}\right)  \geq\left\langle \xi,H\right\rangle .
\]
Taking sup over $\xi$ yields $I\left(  \widetilde{M}\right)  \geq\left\vert
H\right\vert .$
\end{proof}

\bigskip

\bigskip We claim that equality holds: $I\left(  \widetilde{M}\right)
=\left\vert H\right\vert $. For this purpose we need to bring in another
geometric invariant. Since $\widetilde{M}$ admits compact quotients by
discrete subgroups of $G$, the following limit, called the volume entropy,
\[
v=\lim_{r\rightarrow\infty}\frac{\log V\left(  r\right)  }{r},
\]
where $V\left(  r\right)  $ is the volume of the geodesic ball $B\left(
o,r\right)  $ with center $o$ and radius $r$, exists and is independent of the
base point $o$ (cf. \cite{M}). This asymptotic invariant is computed
explicitly in \cite[Appendix C]{BCG} for locally symmetric spaces of
noncompact type. Indeed, the volume form on $\widetilde{M}$ is given by
$\prod\limits_{i}\frac{\sinh\sqrt{\lambda_{i}\left(  \xi\right)  }t}%
{\sqrt{\lambda_{i}\left(  \xi\right)  }}d\sigma_{\xi}dt$ in view of
(\ref{Jacobi}). Therefore
\[
V\left(  r\right)  =\int_{0}^{r}\int_{\mathbb{S}^{m-1}}\prod\limits_{i}%
\frac{\sinh\sqrt{\lambda_{i}\left(  \xi\right)  }t}{\sqrt{\lambda_{i}\left(
\xi\right)  }}d\sigma_{\xi}dt.
\]
From this formula it is easy to derive%
\[
v=\sup_{\xi\in\mathfrak{a}^{+}}\sum_{i}\sqrt{\lambda_{i}\left(  \xi\right)
}=\left\vert H\right\vert .
\]
Now we can prove our main result.

\begin{theorem}
\label{main}Let $\widetilde{M}$ be a symmetric space of noncompact type. Then
we have
\[
I=v=\left\vert H\right\vert .
\]

\end{theorem}

\begin{proof}
Given Lemma \ref{low}, it remains to prove $I\leq v$. This is well known.
Indeed, by the definition of $I$ we have%
\[
V^{\prime}\left(  r\right)  =A\left(  r\right)  \geq IV\left(  r\right)  ,
\]
where $A\left(  r\right)  $ is the area of the surface of the geodesic ball
$B\left(  o,r\right)  $. Integrating gives $V\left(  r\right)  \geq V\left(
1\right)  \exp\left(  Ir\right)  $. It follows that $v\geq I$.
\end{proof}

As a corollary we get the following result originally proved by
Olshanetski\cite{O} (see the discussion in \cite[Appendix C]{BCG}).

\begin{theorem}
Let $\widetilde{M}$ be a symmetric space of noncompact type. Then its bottom
of spectrum $\lambda_{0}\left(  \widetilde{M}\right)  $ is given by the
formula%
\[
\lambda_{0}\left(  \widetilde{M}\right)  =\frac{1}{4}\left\vert H\right\vert
^{2}.
\]

\end{theorem}

\begin{proof}
We recall another fundamental inequality%
\[
\lambda_{0}\left(  \widetilde{M}\right)  \leq\frac{1}{4}v^{2}.
\]
This follows easily by taking test functions $f=\exp\left[  (v+\varepsilon
)r/2\right]  $, with $r$ being the distance function to $o$ and $\varepsilon
>0$, in (\ref{lam0}) and then letting $\varepsilon\rightarrow0$. Combining the
above inequality and (\ref{Cheeger}) we can write%
\[
\frac{1}{4}I^{2}\leq\lambda_{0}\left(  \widetilde{M}\right)  \leq\frac{1}%
{4}v^{2}.
\]
From this we obtain the desired identity from Theorem \ref{main}.
\end{proof}

\end{document}